\documentclass[11pt]{amsart}
\usepackage{tikz-cd}
\usepackage{verbatim}
\usepackage[colorinlistoftodos]{todonotes}
\usepackage{amssymb}
\usepackage{mathtools}
\usepackage[bookmarks,colorlinks,breaklinks]{hyperref}  
\usepackage[top=2cm, bottom=4.5cm, left=3cm, right=3cm]{geometry}
\hypersetup{linkcolor=blue,citecolor=blue,filecolor=blue,urlcolor=blue} 
\usepackage{tikz}\usetikzlibrary{cd,fit,matrix,arrows,decorations.pathmorphing}
\tikzset{commutative diagrams/.cd}
\usepackage{bm}

\numberwithin{equation}{section}

\newtheorem{theorem}{Theorem}[subsection]

\newtheorem{lemma}[theorem]{Lemma}
\newtheorem{proposition}[theorem]{Proposition}
\newtheorem{conjecture}[theorem]{Conjecture}

\theoremstyle{definition}

\newtheorem{definition-theorem}[theorem]{Definition-Theorem}
\newtheorem{example}[theorem]{Example}

\newtheorem{remark}[theorem]{Remark}

\usepackage{enumitem}
\setenumerate[1]{label=\textup{(\alph*)}}
\setenumerate[2]{label=\textup{(\roman*)}}

\newcommand{\N}{{\mathbb N}}

\newcommand{\Q}{{\mathbb Q}}
\newcommand{\C}{{\mathbb C}}
\newcommand{\A}{{\mathbb A}}
\newcommand{\OO}{\mathcal{O}}

\newcommand{\F}{{\mathbb F}}
\newcommand{\Fp}{{\mathbb F}_p}

\newcommand{\Fpbar}{{\overline{\mathbb F}_p}}

\newcommand{\Kbar}{{\overline{K}}}

\newcommand{\hhat}{\widehat{h}} %\hat is too narrow; avoid
\newcommand{\lra}{\longrightarrow}

\begin{document}
\title{Intersection of orbits for  polynomials in characteristic $p$}

\author{Simone Coccia}
\address{Department of Mathematics \\ University of Toronto \\ 40 St. George St. \\ Toronto, ON \\ Canada, M5S 2E4}
\email{simone.coccia@utoronto.ca}

\author{Dragos Ghioca}
\address{Department of Mathematics \\ University of British Columbia \\ 1984 Mathematics Road \\ Canada V6T 1Z2}
\email{dghioca@math.ubc.ca}

\author{Jungin Lee}
\address{Department of Mathematics\\ Ajou University\\ Suwon 16499\\ Republic of Korea}
\email{jileemath@ajou.ac.kr}

\author{Gyeonghyeon Nam}
\address{Department of Mathematics\\ Ajou University\\ Suwon 16499\\ Republic of Korea}
\email{ghnam@ajou.ac.kr}

%\date{\today}

%\bibliographystyle{abbrv}

\begin{abstract}
In \cite{GTZ1, GTZ2}, the following result was established: given polynomials $f,g\in\C[x]$ of degrees larger than $1$, if there exist $\alpha,\beta\in\C$ such that their corresponding orbits $\OO_f(\alpha)$ and $\OO_g(\beta)$ (under the action of $f$, respectively of $g$) intersect in infinitely many points, then $f$ and $g$ must share a common iterate, i.e., $f^m=g^n$ for some $m,n\in\N$. If one replaces $\C$ with a field $K$ of characteristic $p$, then the conclusion fails; we provide numerous examples showing the complexity of the problem over a field of positive characteristic. We advance a modified conjecture regarding polynomials $f$ and $g$ which admit two orbits with infinite intersection over a field of characteristic $p$. Then we present various partial results, along with connections with another deep conjecture in the area, the dynamical Mordell-Lang conjecture.  
\end{abstract}

\maketitle

\section{Introduction}
\label{sec:intro}

We start by setting up some basic notation for our paper in Subsection~\ref{subsec:notation}.

%%%%%%%%%%%%%%%%%%%%%%%%%%%%%%%%%%%%%%%%%%%%%%%%%%%%%%%%%%%%%%%%%%%%%%%%%%%%%%

\subsection{Notation}
\label{subsec:notation}

Throughout this paper, we denote by $\N_0:=\N\cup\{0\}$ the set of non-negative integers. Now, given a self-map $f$ on some quasiprojective variety $X$, we denote by $f^n$ its $n$-th compositional power; by convention, $f^0$ represents the identity map ${\rm id}_X$ on $X$. The orbit of a point $x\in X$ under $f$ is the set of all points $f^n(x)$ for $n\ge 0$; we  denote this orbit by $\OO_f(x)$. A \emph{preperiodic} point $x\in X$ for $f$ is a point with a finite orbit, i.e., it has the property that $f^m(x)=f^n(x)$ for some $0\le m<n$.

Given a field $K$, we always denote by $\Kbar$ a fixed algebraic closure of it. Now, given a polynomial $f\in K[x]$, a conjugate of it is any polynomial of the form $\mu^{-1}\circ f\circ \mu$, where $\mu\in\Kbar[x]$ is some linear polynomial. For any $a\in \Kbar$, we denote by $\tau_a$ the linear polynomial $\tau_a(x):=x+a$.
For a field $K$ of positive characteristic, a polynomial $f\in K[x]$ is called \emph{additive} if $f(x+y)=f(x)+f(y)$ in $K[x,y]$.

%%%%%%%%%%%%%%%%%%%%%%%%%%%%%%%%%%%%%%%%%%%%%%%%%%%%%%%%%%%%%%%%%%%%%%%%%%%%%%

\subsection{The problem of intersection of orbits over fields of characteristic $0$}
In \cite{GTZ1, GTZ2}, it was shown that if $f,g\in\C[x]$ are polynomials of degrees larger than $1$ with the property that there exist $\alpha,\beta\in\C$ such that $\OO_f(\alpha)\cap\OO_g(\beta)$ is infinite, then there exist $m,n\in\N$ such that $f^m=g^n$. The result of \cite{GTZ1} was the first major result supporting the dynamical Mordell-Lang conjecture. More precisely, the dynamical Mordell-Lang conjecture predicts that given a quasiprojective variety $X$ defined over a field $K$ of characteristic $0$, given a point $\alpha\in X(K)$, given a regular self-map $\Phi:X\lra X$, and given a subvariety $Y\subseteq X$, then the \emph{return set}:
\begin{equation}
\label{eq:return_set}
R_{X,\Phi,\alpha,Y}:=\left\{n\in\N_0\colon \Phi^n(\alpha)\in Y\right\} 
\end{equation}
is a union of finitely many sets of the form $\{ak+b\}_{k\in\N_0}$ for some given $a,b\in\N_0$. (Note that if $a>0$, then the set $\{ak+b\}_{k\in\N_0}$ is an infinite arithmetic progression, while if $a=0$, then the aforementioned set is a singleton.) The dynamical Mordell-Lang conjecture sparked a wide interest, leading to several partial results, along with numerous connections with other important questions in arithmetic dynamics. However, the conjecture still remains open in its full generality and no counterexamples are known; for a comprehensive discussion about it, we refer the reader to the book~\cite{BGT-book}.

The result of \cite{GTZ1} can be interpreted as follows: the dynamical action of a polynomial $f$ with complex coefficients is uniquely identified by \emph{any} infinite subset of \emph{any} orbit of $f$. Indeed, for any self-map $\Phi$ on any quasiprojective variety $X$, the main dynamical features of $\Phi$ (such as its set of preperiodic points, for example) are unchanged when replacing $\Phi$ by an iterate. The results of \cite{GTZ1, GTZ2} show that only polynomials which have a common iterate (and hence induce the same dynamical action) can have orbits which intersect in infinitely many points. Variants of the problem of intersection of orbits were studied for other dynamical systems (see \cite{intersection-MRL, intersection-4, intersection-2, Wang}). Furthermore, all these questions fit under the large umbrella of the unlikely intersection problem in arithmetic dynamics; for an  introduction to the topic of unlikely intersections from  classical arithmetic geometry, we refer the reader to the excellent book~\cite{Umberto}.

%%%%%%%%%%%%%%%%%%%%%%%%%%%%%%%%%%%%%%%%%%%%%%%%%%%%%%%%%%%%%%%%%%%%%%%%%%%%%%%

\subsection{The picture in positive characteristic}

If one studies the problem of intersection of orbits of polynomials $f$ and $g$ over a field $K$ of characteristic $p>0$, then there are numerous counterexamples to the strong conclusion from \cite{GTZ1, GTZ2} that $f$ and $g$ must share a common iterate. We present below one such example; more complicated examples appear in Section~\ref{sec:examples}.
\begin{example}
\label{ex:1}
Let $f,g\in\F_2(t)[x]$ be the polynomials $f(x)=x^2+x$ and $g(x)=x^2+t^2+t$. Then a simple induction shows that
\begin{equation}
\label{eq:101}
g^m(0)=t^{2^m}+t\text{ for each }m\ge 1.
\end{equation}
Also, it is easy to see that 
$$f^m(x)=\sum_{i=0}^m \binom{m}{i} x^{2^i}\text{ for each }m\ge 1,$$
which implies that
\begin{equation}
\label{eq:102}
f^{2^k}(t)=t^{2^{2^k}}+t\text{ for each }k\ge 0.
\end{equation}
Equations~\eqref{eq:101}~and~\eqref{eq:102} show that $\OO_f(t)\cap\OO_g(0)$ is infinite. 
However, one sees that no iterate of $f$ equals an iterate of $g$ since $g^m(0) = t^{2^m}+t$ for each $m \ge 1$, while $f^m(0) = 0$ for each $m \ge 1$.
\end{example}

We believe the following modification of the conclusion from \cite{GTZ1, GTZ2} holds for polynomials over fields of positive characteristic whose orbits intersect in infinitely many points.

\begin{conjecture}
\label{conj:main}
Let $K$ be a field of characteristic $p$ and let $f,g\in K[x]$ be polynomials of degrees larger than $1$. If there exist $\alpha,\beta\in K$ with the property that $\OO_f(\alpha)\cap\OO_g(\beta)$ is infinite, then at least one of the following two conditions holds:
\begin{itemize}
\item[(1)] there exist $m,n\in\N$ such that $f^m=g^n$; 
\item[(2)] there exist linear polynomials $\lambda, \mu \in\Kbar[x]$ and additive polynomials $\tilde{f},\tilde{g}\in\Kbar[x]$ such that 
\begin{equation}
\label{eq:condition_additive_1}
f=\lambda^{-1} \circ \tilde{f}\circ \lambda \text{ and }g=\mu^{-1} \circ \tilde{g}\circ \mu.
\end{equation}
Furthermore, there exists $m\in\N$ such that $\tilde{f}^m\circ \tilde{g}^m=\tilde{g}^m\circ \tilde{f}^m$.
\end{itemize}
\end{conjecture}

First, we note that (similar to the problem studied in \cite{GTZ1, GTZ2}), one needs to assume the polynomials $f$ and $g$ have degrees larger than $1$ since otherwise there are many more classes of examples of polynomials which have orbits with infinite intersections; for more details, see Subsection~\ref{subsec:ex_deg_1}.

In our paper we provide various heuristics in support of Conjecture~\ref{conj:main} (for more details about our plan, see Subsection~\ref{subsec:remarks}). We also show that Conjecture~\ref{conj:main} reduces to the case of polynomials of same degree; in particular, we prove the following result.

\begin{theorem}
\label{thm:same_degree_0}
Let $K$ be a field of characteristic $p$ and let $f,g\in K[x]$ be polynomials of degrees $d\ge 2$ and $e\ge 2$, respectively. If there exist no positive integers $r$ and $s$ such that $d^r=e^s$, then for any $\alpha,\beta\in K$, we have that $\OO_f(\alpha)\cap\OO_g(\beta)$ is finite.
\end{theorem} 

We note that conclusion~(2) from Conjecture~\ref{conj:main} is met by all examples we found of polynomials $f$ and $g$ with orbits which intersect in infinitely many points and for  which condition~(1) does not hold (for more details, see Section~\ref{sec:examples}). Also, our examples suggest that perhaps an even stronger conclusion than the commutation of some iterates of $f$ and respectively $g$ must hold; however, the variety of our examples did not offer a clear option for such a stronger condition. Finally,  conclusion~(2) alone would not always guarantee the existence of two orbits with infinite intersection (under the action of two commuting additive polynomials $f$ and $g$, for example), but it seems difficult to find a necessary and sufficient condition to replace the current conclusion~(2) in Conjecture~\ref{conj:main}.

%%%%%%%%%%%%%%%%%%%%%%%%%%%%%%%%%%%%%%%%%%%%%%%%%%%%%%%%%%%%%%%%%%%%%%%%%%%%%%%

\subsection{Further remarks}
\label{subsec:remarks}

We first note that Conjecture~\ref{conj:main} fits into a recent series of papers searching for analog statements over fields of positive characteristic for some of the most important results from the past $20$ years in arithmetic dynamics over fields of characteristic $0$ (see \cite{G-Sina, X, X-Y}, for example). Two of the most active areas of recent research in arithmetic dynamics have been the unlikely intersection principle and the dynamical Mordell-Lang conjecture. There are only a few papers dealing with the unlikely intersection principle for dynamical systems in characteristic $p$ (see \cite{B-Masser, B-Masser-2, G-2, GH}, for example); similarly, there are only a few articles studying  the dynamical Mordell-Lang conjecture in characteristic $p$ (see \cite{JIMJ, G, L-N}, for example). Both of these problems turn out to be \emph{harder} in positive characteristic; for example, in \cite{JIMJ}, it is proven that some special case of the dynamical Mordell-Lang conjecture in characteristic $p$ is already \emph{equivalent} with a deep question regarding polynomial-exponential equations \emph{in characteristic $0$}. Our Conjecture~\ref{conj:main} is both a question about unlikely intersections and it is also directly connected to the dynamical Mordell-Lang conjecture in characteristic $p$ (see Section~\ref{sec:proofs_2} for more details). 

We discuss now some of the difficulties one encounters when trying to employ the strategy from \cite{GTZ1, GTZ2}  for solving Conjecture~\ref{conj:main}. The main two ingredients of the proofs from \cite{GTZ1, GTZ2} are coming from Ritt's theory \cite{Ritt} for polynomial decomposition and from Bilu-Tichy's explicit description \cite{Bilu-Tichy} of all pairs of polynomials $(f,g)$ defined over a number field with the property that the curve $f(x)=g(y)$ contains infinitely many integral points. Both of these two ingredients are missing when working over a field of characteristic $p$. Part of Ritt's theory holds (see \cite{Ritt-p}) assuming \emph{all} field extensions generated by the given polynomials are separable (which would be the case when dealing with polynomials of degrees \emph{less than $p$}), but in general, there are numerous complications arising from additive polynomials. As for Bilu-Tichy's theory, even though there is a classification of all pairs of polynomials $(f,g)$ for which $f(x)-g(y)$ has a factor of degree at most $2$ (see \cite{Bilu-Tichy-p}), in characteristic $p$ we also have that isotrivial curves contain infinitely many integral points (see \cite{Grauert}). Combining the results of \cite{Bilu-Tichy-p} and \cite{Grauert} leads to a very complicated picture in characteristic $p$ for curves of the form $f(x)=g(y)$ which contain infinitely many integral points. For example, as noted in \cite[Section~5,~p.~1581]{Tom}, each curve of the form $f(x)=g(y)$ is already isotrivial if $f(x)$ is an additive, separable polynomial; once again, this motivates treating as \emph{special} all additive polynomials (see conclusion~(2) from our Conjecture~\ref{conj:main}). Furthermore, in Section~\ref{sec:proofs_2} we connect Conjecture~\ref{conj:main} to a special case of the dynamical Mordell-Lang conjecture in characteristic $p$, which is known to be notoriously difficult (as shown in \cite{JIMJ}). All these observations lead us to believe Conjecture~\ref{conj:main} is significantly more difficult than the results obtained in \cite{GTZ1, GTZ2}.

%%%%%%%%%%%%%%%%%%%%%%%%%%%%%%%%%%%%%%%%%%%%%%%%%%%%%%%%%%%%%%%%%%%%%%%%%%%%%%%

\subsection{Plan for our paper}

We start by providing ample examples in Section~\ref{sec:examples} of pairs of polynomials $(f,g)$ (defined over fields of positive characteristic) which have orbits with infinite intersection even though $f$ and $g$ do not share a common iterate. In all our examples, the polynomials $f$ and $g$ are conjugates of suitable additive polynomials and they satisfy condition~(2) from Conjecture~\ref{conj:main}.

We continue by proving in Section~\ref{sec:proofs} (using an argument similar to the one from \cite[Section~8]{GTZ2}) that we may assume in Conjecture~\ref{conj:main} that $f$ and $g$ have the same degree (see Theorem~\ref{thm:same_degree}, which is a strengthening of Theorem~\ref{thm:same_degree_0}).  Furthermore, this allows us to connect the problem for intersecting orbits with the dynamical Mordell-Lang conjecture in characteristic $p$; this is detailed in Section~\ref{sec:proofs_2}. In particular, as explained in Remark~\ref{rem:p-sets}, this provides a motivation for condition~(2) from Conjecture~\ref{conj:main}.

\bigskip

{\bf Acknowledgments.} S. Coccia and D. Ghioca were partially supported by an NSERC Discovery grant. J. Lee and G. Nam were supported by the National Research Foundation of
Korea (NRF) grant funded by the Korea government (MSIT) (No.
RS-2024-00334558).

%%%%%%%%%%%%%%%%%%%%%%%%%%%%%%%%%%%%%%%%%%%%%%%%%%%%%%%%%%%%%%%%%%%%%%%%%%%%%%%
%%%%%%%%%%%%%%%%%%%%%%%%%%%%%%%%%%%%%%%%%%%%%%%%%%%%%%%%%%%%%%%%%%%%%%%%%%%%%%%

\section{Examples of orbits of polynomials with infinite intersection}
% which intersect in infinitely many points
\label{sec:examples}

From now on, we let $K$ be a field of characteristic $p$, $\Kbar$ be a fixed algebraic closure of $K$ and $\Fpbar$ be the algebraic closure of $\Fp$ inside $\Kbar$.

We start by proving the following easy fact which will be used throughout our paper. 
\begin{lemma}
\label{lem:affine}
If $f(x) \in K[x]$ is an additive polynomial of degree larger than $1$, then for any $\gamma\in K$, the polynomial $f_1(x):=f(x)+\gamma$ is a conjugate of $f(x)$.
\end{lemma}

\begin{proof} 
Choose $\delta\in\Kbar$ such that $f(\delta) - \delta=\gamma$. Then we have $\tau_{-\delta}\circ f\circ \tau_\delta = f_1$.
\end{proof}

%%%%%%%%%%%%%%%%%%%%%%%%%%%%%%%%%%%%%%%%%%%%%%%%%%%%%%%%%%%%%%%%%%%%%%%%%%%%%%%

\subsection{The orbit intersection problem when one of the polynomials is linear}
\label{subsec:ex_deg_1}
In this Subsection, we show that there are numerous examples of pairs of polynomials $(f,g)$ with $f$ linear having the property that they have orbits which intersect in infinitely many points.

\begin{example}
\label{ex:simple_linear}
For some $t\in K \setminus \Fpbar$ and an integer $r>1$, we let $f(x)=tx$ and $g(x)=x^r$. Then clearly $\OO_f(1)\cap\OO_g(t)$ is infinite.
\end{example}

However, we can construct more elaborate examples involving any given linear polynomial $f(x)=t x + \delta$, as long as $t\notin\Fpbar$. On the other hand, we also note that whenever $\gamma\in\Fpbar$, the linear polynomial $f(x)=\gamma x+\delta$ has finite order, i.e., all its orbits are finite.

\begin{example}
\label{ex:difficult_linear}
We let $f(x)=tx + \delta$ (for some $t,\delta\in K$, with $t\notin\Fpbar$) and also, we let $\epsilon:=\frac{\delta}{t-1}$. Then letting $\alpha:=1-\epsilon$, we get that
\begin{equation}
\label{eq:4}
\OO_f(\alpha)=\left\{t^n-\epsilon\colon n\ge 0 \right\}.
\end{equation}
We let $g(x)=\tau_{-\epsilon}\circ x^r\circ \tau_\epsilon$ (for any integer $r\ge 2$). Then letting $\beta:=t-\epsilon$, we get
\begin{equation}
\label{eq:3}
\OO_g(\beta)=\left\{t^{r^n}-\epsilon\colon n\ge 0\right\}.
\end{equation}
Equations~\eqref{eq:4}~and~\eqref{eq:3} show that $\OO_f(\alpha)\cap\OO_g(\beta)$ is infinite.
\end{example}

%%%%%%%%%%%%%%%%%%%%%%%%%%%%%%%%%%%%%%%%%%%%%%%%%%%%%%%%%%%%%%%%%%%%%%%%%%%%%%%

\subsection{Examples showing that condition~(2) is essential for Conjecture~\ref{conj:main}}
\label{subsec:the_examples}

Our next example is a vast generalization of Example~\ref{ex:1}. 
\begin{example}
\label{ex:1_general}
We let $K:=\Fp(t)$, $f(x) \in K[x]$ be an additive polynomial of degree larger than $1$ and $g(x):=f(x)+x+f(t)$. A simple induction on $m$ yields that
\begin{equation}
\label{eq:15}
g^m(0)=\sum_{i=1}^{m} \binom{m}{i}\cdot f^i(t).
\end{equation} 
Therefore, for each $k\ge 1$ we have
\begin{equation}
\label{eq:16}
g^{p^k}(0)=f^{p^k}(t).
\end{equation}
In particular, this means that $\OO_f(t)\cap \OO_g(0)$ is infinite. Also, note that $g(x)$ is a conjugate of the additive polynomial $\tilde{g}(x):=f(x)+x$ (see Lemma~\ref{lem:affine}); furthermore, it is immediate to see that $f(x)$ commutes with $\tilde{g}(x)$. On the other hand, $f(x)$ and $g(x)$ normally share no common iterate; this is immediately seen if $f\in\Fpbar[x]$ since $g(x)$ is not defined over $\Fpbar$, but it is also true for almost any choice of additive polynomial $f(x)$ (see equation~\eqref{eq:15}).
\end{example}

\begin{remark}
\label{rem:p-sets_0}
Interestingly, in Example~\ref{ex:1_general}, the return set
\begin{equation}
\label{eq:17}
R:=R_{f,t,g,0}:=\left\{(m,n) \in \N_0 \times \N_0 \colon f^m(t)=g^n(0)\right\} 
\end{equation}
is the set of all pairs of integers $(p^k,p^k)$ for each $k\ge 0$. The fact that the return set $R$ from equation~\eqref{eq:17} consists of pairs of integers involving powers of $p$ is a common feature for all of our examples of polynomials $(f,g)$ with two orbits intersecting infinitely often, even though the polynomials do not share a common iterate. We believe \emph{all possible} examples of such polynomials $(f,g)$ have a return set similar to the one from equation~\eqref{eq:17} (see Section~\ref{sec:proofs_2}, especially Remark~\ref{rem:p-sets}).
\end{remark}

The following example shows that condition~(2) from Conjecture~\ref{conj:main} cannot be strengthened by asking that the two corresponding additive polynomials commute.
\begin{example}
\label{ex:non-commute}
Let $r\ge 2$ be an integer, $\lambda\in \F_{p^r}$ be a generator for the field extension $\F_{p^r}/\F_p$ and $K:=\F_{p^r}(t)$. Also let $f(x)=x^p$ and $g = \tau_{-\lambda t} \circ \tilde{g} \circ \tau_{\lambda t}$ for
$\tilde{g}(x)=\lambda x + x^{p^r}$. Then for any $m\ge 1$, we have
\begin{equation}
\label{eq:11}
\tilde{g}^m(x)=\sum_{i=0}^m \binom{m}{i}\cdot \lambda^i\cdot x^{p^{r(m-i)}}.
\end{equation}
So, for any $k\ge 1$, we have that
\begin{equation}
\label{eq:12}
g^{p^{rk}}((1-\lambda)t)=t^{p^{rp^{rk}}} = f^{rp^{rk}}(t).
\end{equation} 
Equation~\eqref{eq:12} shows that $\OO_f(t)\cap \OO_g\left((1-\lambda )t\right)$ is infinite. 
However, no iterate of $f$ equals an iterate of $g$ since $g^m(-\lambda t) = -\lambda t$ for each $m \ge 1$, while $f^m(-\lambda t) = (-\lambda t)^{p^m} \neq -\lambda t$ for each $m \ge 1$.
Furthermore, since $\lambda$ is a generator for $\F_{p^r}/\Fp$, $\tilde{g}$ commutes with $f^r$ but does not commute with $f$.
\end{example}

\begin{example} \label{ex:general}
Let $K$ be a field of characteristic $p$, $f$ and $h$ be additive polynomials over $K$ such that $\deg f > 1$ and $f^m\circ  h=h\circ f^m$ ($m \ge 1$), $\delta \in \Kbar$ and $g := \tau_{\delta} \circ (f^m+h) \circ \tau_{-\delta}$. Then for any $k \ge 0$ and $\alpha \in \Kbar$, we have
\begin{equation} \label{eq:exg1}
g^{p^k}(\alpha + \delta) = f^{mp^k}(\alpha) + h^{p^k}(\alpha) + \delta
\end{equation}
so $g^{p^k}(\alpha + \delta) = f^{mp^k}(\alpha)$ if and only if $h^{p^k}(\alpha) + \delta = 0$. Now assume that $\alpha$ is a preperiodic under $h$, non-preperiodic under $f$ and choose any $\delta \in \Kbar$ such that $h^{p^k}(\alpha) + \delta = 0$ for infinitely many $k \ge 0$. Then $\OO_f(\alpha) \cap \OO_g\left( \alpha+\delta \right)$ is infinite.

We note that this example generalizes both Example \ref{ex:1_general} and \ref{ex:non-commute}.
\begin{enumerate}
    \item If $K=\F_{p}(t)$, $h(x)=x$, $m=1$ and $(\alpha, \delta)=(t, -t)$, we have Example \ref{ex:1_general}.

    \item If $K=\F_{p^r}(t)$, $f(x)=x^{p}$, $h(x) = \lambda x$ ($\lambda \in \F_{p^r}$ is a generator of $\F_{p^r}/\F_p$), $m=r$ and $(\alpha, \delta)=(t, -\lambda t)$, we have Example \ref{ex:non-commute}. In general, if $h(x) = \lambda x$ ($\lambda \in \F_{p^r}$) then we can choose any $f$ and $m$ such that $f^m(x) = a_0x + a_1x^{p^r}+ a_2x^{p^{2r}} + \cdots$.
\end{enumerate}
\end{example}

Now we provide an example which is not a consequence of the general construction shown in Example \ref{ex:general}. We will use the following lemma.

\begin{lemma}\label{lem:nonlinearterm}
    Let $K$ be a field of characteristic $p$ and $t \in K$. Then for every $n \ge 1$,
    \begin{equation}\label{eq:lemproof}
    (t^{p-1}+t^{p-2}+\cdots + t+1)^{\frac{p^n-1}{p-1}} = t^{p^n-1}+t^{{p^n-2}}+\cdots +t+1.
\end{equation}
\end{lemma}
\begin{proof}
We use induction on $n$. Assume that equation~\eqref{eq:lemproof} holds for $n =m-1 \ge 1$. Then
\begin{equation*}
\begin{split}
    &(t^{p-1}+\cdots + t+1)^{\frac{p^{m}-1}{p-1}}\\
    = \, &
(t^{p-1}+\cdots + t+1)^{p^{m-1}} (t^{p-1}+\cdots + t+1)^{\frac{p^{m-1}-1}{p-1}} \\
= \, & (t^{p^{m-1}(p-1)}+ \cdots + t^{p^{m-1}}+1) (t^{p^{m-1}-1}+\cdots +t+1 ) \\
= \, & t^{p^m-1} + t^{p^m-2} + \cdots + t + 1
\end{split}
\end{equation*} 
so equation~\eqref{eq:lemproof} holds for $n = m$. 
\end{proof}

\begin{example}
\label{ex:x^p}
Let $K = \F_p(t)$, $f(x) = x^{p^{p-1}}+x^{p^{p-2}}+\cdots +x^p+x$ and $g(x) = x^p+t$. Since $f^m(x)=\sum_{i=0}^{{m(p-1)}} a_i x^{p^i}$ where $a_i$ is the coefficient of $t^i$ in $(t^{p-1}+ \cdots +t+1)^m$, we have 
\begin{equation} \label{eq:exxp1}
f^{\frac{p^n-1}{p-1}}(x) 
= \sum_{i=0}^{p^n-1} x^{p^i}
\end{equation}
by Lemma \ref{lem:nonlinearterm}. We also have $g^n(x) = x^{p^n} + \sum_{i=0}^{n-1} t^{p^i}$. 
Now choose $\delta \in \Kbar$ such that $\delta^p - \delta = t$. Let $f_1 = \tau_{\delta} \circ f \circ \tau_{- \delta}$ and $g_1 = \tau_{\delta} \circ g \circ \tau_{- \delta}$. Then $g_1(x)=x^p$ and
\begin{equation} \label{eq:exxp2}
f_1^{\frac{p^n-1}{p-1}}(t + \delta) = g_1^{p^n}(\delta) = \sum_{i=0}^{p^n-1} t^{p^i} + \delta
\end{equation}
for every $n \ge 1$, which implies that $\OO_{f_1}(t + \delta) \cap \OO_{g_1}(\delta)$ is infinite.
Note that if $p$ is odd, then $f^r(x)$ is not of the form $h(x) = x^{p^k}+cx$ for every $r \ge 1$. If $f^r(x) = x^{p^k}+cx$, then clearly $c=1$. Now we have $f^r(1) =0$ in $K$, while $h(1) = 2 \neq 0$ in $K$. 
\end{example}

%%%%%%%%%%%%%%%%%%%%%%%%%%%%%%%%%%%%%%%%%%%%%%%%%%%%%%%%%%%%%%%%%%%%%%%%%%%%%%%
%%%%%%%%%%%%%%%%%%%%%%%%%%%%%%%%%%%%%%%%%%%%%%%%%%%%%%%%%%%%%%%%%%%%%%%%%%%%%%%

\section{Reduction of Conjecture~\ref{conj:main} to the case of polynomials of same degree} 
\label{sec:proofs}

One of the common features of all examples of polynomials $(f,g)$ having two orbits with infinite intersection is the fact that the degrees of our polynomials are multiplicative dependent, i.e., $\deg(f)^r=\deg(g)^s$ for some suitable positive integers $r$ and $s$; in other words some iterate of $f$ have the same degree as some iterate of $g$. The existence of this feature is explained in the following theorem, which is the main result of this Section. It is a slight extension of Theorem~\ref{thm:same_degree_0}.

% The main result of this Section is the following theorem, which is a slight extension of Theorem~\ref{thm:same_degree_0}.
\begin{theorem}
\label{thm:same_degree}
Let $K$ be a field of characteristic $p$, let $\alpha,\beta\in K$, let $f,g\in K[x]$ be polynomials of degrees $d\ge 2$ and $e \ge 2$, respectively. If $\OO_f(\alpha)\cap\OO_g(\beta)$ is infinite, then the following conditions must hold:
\begin{itemize}
\item[(i)] there exist integers $r$ and $s$ such that $d^r=e^s$; and 
\item[(ii)] there exist $a,b\in \N_0$ and there exist infinitely many $n\in\N$ such that
\begin{equation}
\label{eq:1001}
f^{rn+a}(\alpha)=g^{sn+b}(\beta).
\end{equation}
\end{itemize}
\end{theorem}

We prove Theorem~\ref{thm:same_degree} in Subsection~\ref{subsec:proof_thm}. Then, in Subsection~\ref{subsec:conseq_thm}, we derive an important reduction in our Conjecture~\ref{conj:main}.

%%%%%%%%%%%%%%%%%%%%%%%%%%%%%%%%%%%%%%%%%%%%%%%%%%%%%%%%%%%%%%%%%%%%%%%%%%%%%%%

\subsection{Proof of Theorem~\ref{thm:same_degree}}
\label{subsec:proof_thm}

We work with the notation and hypotheses from Theorem~\ref{thm:same_degree}. 

First of all, we can reduce to the case where $K$ is a finitely generated extension of  $\Fpbar$ since we can always replace $K$ with the field generated (over $\Fpbar$) by the coefficients of $f$ and $g$, and also by $\alpha$ and $\beta$. 

Note that  $K$ cannot be $\Fpbar$, since this would mean that $f,g\in\Fpbar[x]$ and also $\alpha,\beta\in\Fpbar$, which would make both $\alpha$ and $\beta$ preperiodic points for $f$, respectively $g$; thus, their orbits could not have infinite intersection. Therefore, $K$ is a function field of finite transcendence degree over $\Fpbar$.

We let $V$ be a projective variety defined over $\Fpbar$, regular in codimension $1$, whose function field is $K$. Then we let $\Omega_V$ be the set places of $K$ corresponding to the irreducible divisors of $V$. We have that $K$ is a product formula field with respect to places from $\Omega_V$ and so, we can construct the Weil height $h:K\lra \Q_{\ge 0}$ with respect to the absolute values from $\Omega_V$; for more details, we refer the reader to \cite[Chapter~3,~Section~3]{Lang}. We also construct the canonical heights with respect to the polynomials $f$ and $g$ (see \cite[Section~1]{Silverman} for more details); so, we let
\begin{equation}
\label{eq:canonical_height}
\hhat_f(\gamma)=\lim_{n\to\infty}\frac{h\left(f^n(\gamma)\right)}{\deg(f)^n}\text{ respectively, }\hhat_g(\gamma)=\lim_{n\to\infty}\frac{h\left(g^n(\gamma)\right)}{\deg(g)^n}
\end{equation}
be the canonical heights of any point $\gamma\in K$ for the action of $f$, respectively of $g$. As proven in \cite[Theorem~1.3]{rationality}, both $\hhat_f(\gamma)$ and $\hhat_g(\gamma)$ are non-negative rational numbers.

As shown in \cite{Rob} (see also \cite{Baker} for a more general result), since $\alpha$ and $\beta$ are non-preperiodic points for $f$ and respectively, for $g$ (because their corresponding orbits are infinite), then 
\begin{equation}
\label{eq:1010}
u_1:=\hhat_f(\alpha)>0\text{ and }u_2:=\hhat_g(\beta)>0.
\end{equation}
Indeed, \cite{Rob, Baker} prove that if the canonical height of a point (under the action of a polynomial) equals $0$ then the point is preperiodic, or the point along with the polynomial  are defined over the constant field (after a suitable linear conjugation). Now, in our case, since the constant field is $\Fpbar$, even if the polynomials $f$ or $g$ were defined over $\Fpbar$ (after some suitable conjugation), each point in $\Fpbar$ would be preperiodic, thus showing that indeed, each non-preperiodic point must have positive canonical height.

Now, we know from \cite[Theorem~1.1]{Silverman} that there exist some positive (real) constants $c_f$ and $c_g$ with the property that for each $\gamma\in K$, we have 
\begin{equation}
\label{eq:1011}
\left|h(\gamma) - \hhat_f(\gamma)\right|<c_f\text{ and }\left|h(\gamma)-\hhat_g(\gamma)\right|<c_g.
\end{equation}
The constants $c_f$ and $c_g$ depend on the respective polynomials $f$ and $g$ (and also on the set $\Omega_V$) but they are independent of the point $\gamma$. So, equation~\eqref{eq:1011} yields that (letting $c:=c_f+c_g$) we have
\begin{equation}
\label{eq:1012}
\left|\hhat_f(\gamma)-\hhat_g(\gamma)\right|<c\text{ for each }\gamma\in K.
\end{equation}
We apply inequality~\eqref{eq:1012} to each point $\gamma\in \OO_f(\alpha)\cap\OO_g(\beta)$. So, for each pair $(m,n)$ of positive integers for which $f^m(\alpha)=g^n(\beta)$, we have (according to inequality~\eqref{eq:1012}) that
\begin{equation}
\label{eq:1013}
\left|\hhat_f\left(f^m(\alpha)\right)-\hhat_g\left(g^n(\beta)\right)\right|<c.
\end{equation}
We know from \cite{Silverman} that $\hhat_f\left(f^m(\alpha)\right)=d^m\cdot \hhat_f(\alpha)$ and also, $\hhat_g\left(g^n(\beta)\right)=e^n\cdot \hhat_g(\beta)$ (which follows from equation~\eqref{eq:canonical_height}). So, using equations~\eqref{eq:1010}~and~\eqref{eq:1013}, we obtain that for each pair of integers $(m,n)$ corresponding to a point $\gamma=f^m(\alpha)=g^n(\beta)$, we have
\begin{equation}
\label{eq:1014}
\left|d^m\cdot u_1-e^n\cdot u_2\right|<c.
\end{equation}
Now, since $u_1$ and $u_2$ are positive rational numbers, it means that there exists some positive integer $D$ such that $\tilde{u}_1:=Du_1$ and $\tilde{u}_2:=Du_2$ are positive integers. So, we have the inequality:
\begin{equation}
\label{eq:1015}
\left|d^m\cdot \tilde{u}_1-e^n\cdot\tilde{u}_2\right|<cD,
\end{equation}
for each pair $(m,n)$ of positive integers corresponding to some point $\gamma$ from the intersection of the two orbits. Because there exist finitely many integers of absolute value less than $cD$, while $\OO_f(\alpha)\cap\OO_g(\beta)$ is infinite, we conclude that there exists some integer $v$ and there exist infinitely many pairs $(m,n)\in\N\times\N$ (such that $f^m(\alpha)=g^n(\beta)$) for which
\begin{equation}
\label{eq:1016}
d^m\cdot \tilde{u}_1-e^n\cdot \tilde{u}_2=v.
\end{equation}
Now, equation~\eqref{eq:1016} has finitely many integer solutions $(m,n)$, unless $v=0$; this is a special case of the famous Mordell-Lang theorem for algebraic tori proved by Laurent \cite{Laurent}. So, we must have that $v=0$ and then applying equation~\eqref{eq:1016} for two different pairs of integers $(m_1,n_1)$ and $(m_2,n_2)$ (with $m_2>m_1$ and thus, also $n_2>n_1$), we get that
\begin{equation}
\label{eq:1017}
d^{m_2-m_1}=e^{n_2-n_1}.
\end{equation}  
So, letting $r=\frac{m_2-m_1}{\gcd(m_2-m_1,n_2-n_1)}$ and $s:=\frac{n_2-n_1}{\gcd(m_2-m_1,n_2-n_1)}$, we obtain the desired conclusion~(i) in Theorem~\ref{thm:same_degree}. Furthermore, by our definition for $r$ and $s$ (which are coprime), we have that whenever $d^{\ell_1}=e^{\ell_2}$ for some positive integers $\ell_1$ and $\ell_2$, then we must have that
\begin{equation}
\label{eq:1020}
\ell_1=k\cdot r\text{ and }\ell_2=k\cdot s\text{ for some suitable }k\in\N.
\end{equation}

Now, for all but finitely many pairs $(m,n)$ corresponding to points $\gamma=f^m(\alpha)=g^n(\beta)$, we must have that 
\begin{equation}
\label{eq:1018}
d^m\cdot \tilde{u}_1=e^n\cdot \tilde{u}_2.
\end{equation}
Fix some pair $(m_0,n_0)$ of integers satisfying  equation~\eqref{eq:1018}. Then for any other pair of integers  $(m,n)$ satisfying both equation~\eqref{eq:1018} along with the inequality $m>m_0$, equation~\eqref{eq:1020} yields that there exists a positive integer $k$ such that $m=m_0+kr$ and $n=n_0+k s$. This provides the desired condition~(ii) and allows us to conclude our proof for Theorem~\ref{thm:same_degree}.

%%%%%%%%%%%%%%%%%%%%%%%%%%%%%%%%%%%%%%%%%%%%%%%%%%%%%%%%%%%%%%%%%%%%%%%%%%%%%%%

\subsection{Key reduction for Conjecture~\ref{conj:main}}
\label{subsec:conseq_thm}

Theorem~\ref{thm:same_degree} allows us to reduce Conjecture~\ref{conj:main} to the following special case of it.

\begin{conjecture}
\label{conj:special}
Let $K$ be a field of characteristic $p$ and let $f,g\in K[x]$ be polynomials of the same degree $d\ge 2$. If there exist $\alpha,\beta\in K$ with the property that for infinitely many $n\in\N$, we have 
\begin{equation}
\label{eq:1002}
f^n(\alpha)=g^n(\beta), 
\end{equation}
then at least one of the following two conditions must hold:
\begin{itemize}
\item[(1)] there exists $m\in\N$ such that $f^m=g^m$; 
\item[(2)] there exist linear polynomials $\lambda, \mu \in\Kbar[x]$ and there exist additive polynomials $\tilde{f},\tilde{g}\in\Kbar[x]$ such that 
\begin{equation}
	\label{eq:condition_additive_2}
	f=\lambda^{-1} \circ \tilde{f}\circ \lambda \text{ and }g=\mu^{-1} \circ \tilde{g}\circ \mu.
\end{equation}
Furthermore, there exists $m\in\N$ such that $\tilde{f}^m\circ \tilde{g}^m=\tilde{g}^m\circ \tilde{f}^m$.
\end{itemize}
\end{conjecture}

\begin{proposition}
\label{prop:main_to_special}
Conjectures~\ref{conj:main} and \ref{conj:special} are equivalent.
\end{proposition}

\begin{proof}
First of all, clearly, Conjecture~\ref{conj:special} is the special case of Conjecture~\ref{conj:main} when the two polynomials $f$ and $g$ have the same degree. So, it remains to show that assuming the validity of Conjecture~\ref{conj:special}, then we can recover the desired conclusion in Conjecture~\ref{conj:main}.

So, we let $f,g\in K[x]$ be polynomials of degrees $d\ge 2$, respectively $e\ge 2$ satisfying the hypotheses of Conjecture~\ref{conj:main} for some points $\alpha$ and $\beta$ (non-preperiodic for $f$, respectively $g$). Since $\OO_f(\alpha)\cap\OO_g(\beta)$ is infinite, then Theorem~\ref{thm:same_degree} yields that conclusions~(i)~and~(ii) must hold. In particular, for some positive integers $r$ and $s$, we have $d^r=e^s$ and furthermore, there exist infinitely many $n\in\N$ such that
\begin{equation}
\label{eq:1003}
f^{rn+a}(\alpha)=g^{sn+b}(\beta).
\end{equation} 
We let $\alpha_1:=f^a(\alpha)$ and $\beta_1:=g^b(\beta)$; also, let $f_1:=f^r$ and $g_1:=g^s$. Then $\deg(f_1)=\deg(g_1)\ge 2$ and equation~\eqref{eq:1003} shows that the hypotheses in Conjecture~\ref{conj:special} are met. Therefore, we know that either conclusion~(1) or conclusion~(2) from Conjecture~\ref{conj:special} must hold. Clearly, if $f_1$ and $g_1$ have a common iterate, this yields that also $f$ and $g$ have a common iterate. 

So, assume next that conclusion~(2) in Conjecture~\ref{conj:special} holds for $f_1$ and $g_1$ and we will show that conclusion~(2) in Conjecture~\ref{conj:main} must hold for $f$ and $g$. In particular, we know that there exist some linear polynomials $\lambda, \mu \in\Kbar[x]$ such that
\begin{equation}
\label{eq:1005}
\tilde{f}_1:=\lambda\circ f_1\circ \lambda^{-1}\text{ and } \tilde{g}_1:=\mu\circ g_1\circ \mu^{-1}
\end{equation}
are additive polynomials. Letting 
$$\tilde{f}:=\lambda\circ f\circ \lambda^{-1}\text{ and }\tilde{g}:=\mu\circ g\circ \mu^{-1},$$
we get (using equation~\eqref{eq:1005} and the fact that $f_1=f^r$, while $g_1=g^s$) that $\tilde{f}^r$ and $\tilde{g}^s$ are additive polynomials. Using \cite[Theorem~4]{Ritt-p}, we obtain that $\tilde{f}$ and $\tilde{g}$ must be additive polynomials, as desired in conclusion~(2) of Conjecture~\ref{conj:main}. Finally, the ``furthermore'' statement in part~(2) of Conjecture~\ref{conj:main} is immediately implied by the ``furthermore'' statement from part~(2) of Conjecture~\ref{conj:special}; indeed, if $\tilde{f}^{rm}$ and $\tilde{g}^{sm}$ commute, then $\tilde{f}^{rsm}$ and $\tilde{g}^{rsm}$ commute as well.

This concludes our proof of Proposition~\ref{prop:main_to_special}.
\end{proof}

%%%%%%%%%%%%%%%%%%%%%%%%%%%%%%%%%%%%%%%%%%%%%%%%%%%%%%%%%%%%%%%%%%%%%%%%%%%%%%
%%%%%%%%%%%%%%%%%%%%%%%%%%%%%%%%%%%%%%%%%%%%%%%%%%%%%%%%%%%%%%%%%%%%%%%%%%%%%%

\section{Connections to the dynamical Mordell-Lang conjecture}
\label{sec:proofs_2}

Proposition~\ref{prop:main_to_special} shows that Conjecture~\ref{conj:main} reduces to its special case from Conjecture~\ref{conj:special}. In other words, it suffices to work in Conjecture~\ref{conj:main} under the extra hypotheses that the polynomials $f$ and $g$ have the same degree and there exist infinitely many $n\in\N$ such that $f^n(\alpha)=g^n(\beta)$. Geometrically, this last condition can be reformulated by asking that the diagonal line in $\A^2$ has infinite intersection with the orbit of the point $(\alpha,\beta)\in\A^2(K)$ under the action of the regular self map on $\A^2$ given by $(x,y)\mapsto (f(x),g(y))$. This allows us to connect our Conjecture~\ref{conj:main} (through its reduction from Conjecture~\ref{conj:special}) to a special case of the dynamical Mordell-Lang conjecture in characteristic $p$. Indeed, the following statement is a special case of the more general question posed in \cite[Conjecture~13.2.0.1]{BGT-book} (see also \cite{G}, which deals specifically with the case of curves in the dynamical Mordell-Lang conjecture).

\begin{conjecture}
\label{conj:DML_p}
Let $K$ be a field of characteristic $p>0$ and let $f,g\in K[x]$. We denote by $\Phi:\A^2\lra \A^2$ be the regular self-map given by
\begin{equation}
\label{eq:coordinatewise_action}
(x,y)\mapsto \left(f(x), g(y)\right).
\end{equation}
Then for any point $\gamma\in \A^2(K)$ and for any curve $C\subset \A^2$ defined over $K$, the return set $R:=\left\{n\in\N_0\colon\Phi^n(\gamma)\in C\right\}$ is a union of finitely many (infinite) arithmetic progressions along with finitely many sets of the form
\begin{equation}
\label{eq:p-set}
\left\{ap^{rk}+b\colon k\in\N_0\right\},
\end{equation}
for some given rational numbers $a$ and $b$ and some given non-negative integer $r$.
\end{conjecture}

\begin{remark}
\label{rem:form_a_b}
In equation~\eqref{eq:p-set}, if $r=0$ then that entire set is just a singleton. 

Next, assume the set from equation~\eqref{eq:p-set} is infinite. Since the elements from a set of the form equation~\eqref{eq:p-set} must all be (non-negative) integers, we immediately conclude that $a$ and $b$ are rational numbers of the form $\frac{a_0}{p^r-1}$, respectively $\frac{b_0}{p^r-1}$ for some integers $a_0>0$ and $b_0$.
\end{remark}

So, we go back to the setting from Conjecture~\ref{conj:special} of two polynomials $f,g\in K[x]$ of same degree (larger than $1$) and two starting points $\alpha,\beta\in K$ for which there exist infinitely many $n\in\N$ such that $f^n(\alpha)=g^n(\beta)$. Then Conjecture~\ref{conj:DML_p} yields that (at least) one of the following two conditions must hold:
\begin{itemize}
\item[(I)] there exist integers $a>0$ and $b\ge 0$ such that for all $n\in\N_0$, we have
\begin{equation}
\label{eq:arith_progr_in}
f^{an+b}(\alpha)=g^{an+b}(\beta);
\end{equation} 
\item[(II)] there exists $r\in\N$ and there exist rational numbers $a>0$ and $b$ such that for each  $n\in\N_0$, we have
\begin{equation}
\label{eq:p-set_in}
f^{ap^{rn}+b}(\alpha)=g^{ap^{rn}+b}(\beta).
\end{equation}
\end{itemize}

\begin{lemma}
\label{lem:easy_arith_progr}
With the above notation for $K$, $f$, $g$, $\alpha$ and $\beta$, if condition~(I) holds, then $f$ and $g$ share a common iterate, i.e., conclusion~(1) from Conjecture~\ref{conj:main} holds.
\end{lemma}

\begin{proof}
Using the hypothesis from condition~(I) above, we get that there exist infinitely many points $\gamma$ on the diagonal line $\Delta\subset \A^2$ (all of the form $(f^{an+b}(\alpha),g^{an+b}(\beta))$), which are mapped by the endomorphism of $\A^2$ given by 
$$(x,y)\mapsto \left(f^a, g^a\right)$$
back on the same diagonal line $\Delta$. This means that $\Delta$ is mapped by $(f^a,g^a)$ back to itself, i.e., $f^a=g^a$, as desired.

This concludes our proof of Lemma~\ref{lem:easy_arith_progr}.
\end{proof}

\begin{remark}
\label{rem:p-sets}
Lemma~\ref{lem:easy_arith_progr} along with Conjecture~\ref{conj:DML_p} further reduces Conjecture~\ref{conj:special} (and therefore Conjecture~\ref{conj:main}) to the special case that condition~(II) from above holds (while condition~(I) is not met). In this case, it is widely believed (see the discussion from \cite[p.~3]{X-Y}) that condition~(II) alone appears in the dynamical Mordell-Lang conjecture only when the dynamical action is induced by an algebraic group action. This heuristic along with all the examples we found (see Section~\ref{sec:examples}) led us to conjecturing the conclusion~(2) from our Conjecture~\ref{conj:main}. 
\end{remark}

%%%%%%%%%%%%%%%%%%%%%%%%%%%%%%%%%%%%%%%%%%%%%%%%%%%%%%%%%%%%%%%%%%%%%%%%%%%%%%%
%%%%%%%%%%%%%%%%%%%%%%%%%%%%%%%%%%%%%%%%%%%%%%%%%%%%%%%%%%%%%%%%%%%%%%%%%%%%%%%

\end{document}